\documentclass[12pt]{article}

\usepackage{authblk}
\usepackage{diagbox}
\usepackage{mathtools}
\usepackage{bbm}
\usepackage{latexsym}
\usepackage{epsfig}
\usepackage{amsmath,amsthm,amssymb,enumerate}
\usepackage{comment}

\usepackage[a-1b]{pdfx}
\usepackage{hyperref}

\usepackage{color}

\def\nn{\nonumber}
\def\a{\alpha} \def\b{\beta}  
    \def\g{\gamma}
\def\G{\Gamma}  
 \def\th{\theta}    
 \def\m{\mu}  
\def\r{\rho}  \def\s{\sigma} 
 \def\om{\omega}

\def\cP{{\cal P}}
\def\cQ{{\cal Q}}

\newtheorem{theorem}{Theorem}
\newtheorem{lemma}[theorem]{Lemma}

\newcommand{\rdown}[1]{{\left\lfloor #1\right \rfloor}}

\newcommand{\brac}[1]{\left(#1\right)}

\newcommand{\bfrac}[2]{\left(\frac{#1}{#2}\right)}

\newcommand{\set}[1]{\left\{#1\right\}}

\def\E{\mathbb{E}}

\def\Pr{\mathbb{P}}

\def\cF{{\cal F}}
\newcommand{\ignore}[1]{}

\def\cC{{\mathcal C}}

\newcommand{\beq}[2]{\begin{equation}\label{#1}#2\end{equation}}
\newcommand{\mults}[1]{\begin{multline*}#1\end{multline*}}

\def\dist{\text{dist}}
\def\SMA{{\bf SMALL}}

\usepackage{tikz}
\usetikzlibrary{decorations.pathmorphing}
\usetikzlibrary{positioning}
\usetikzlibrary{arrows,automata}
\usetikzlibrary{shapes.misc}
\usetikzlibrary{backgrounds}
\usetikzlibrary{arrows,shapes}

\def\bm{{\bf m}}
\def\bM{{\bf M}}
\def\gnpa{G_{n,p}^{\boldsymbol \alpha}}
\def\dnpa{D_{n,p}^{\boldsymbol \alpha}}
\def\gnqa{G_{n,q}^{\boldsymbol \alpha}}
\def\dnpastar{D_{n,p}^{{\boldsymbol \alpha},*}}

\title{Colorful Hamilton cycles in random graphs}
\author[1]{Debsoumya Chakraborti\thanks{Supported by the Institute for Basic Science (IBS-R029-C1)}}
\author[2]{Alan Frieze\thanks{Research supported in part by NSF grant DMS1952285}}
\author[2]{Mihir Hasabnis}
\affil[1]{\small Discrete Mathematics Group, Institute for Basic Science (IBS), Daejeon,~South~Korea}
\affil[2]{\small Department of Mathematical Sciences, Carnegie Mellon University, Pittsburgh,~USA}
\affil[ ]{\small Email:
\texttt{debsoumya@ibs.re.kr},
\texttt{alan@random.math.cmu.edu},
\texttt{mhasabni@andrew.cmu.edu} }

\begin{document}

\maketitle

\begin{abstract}
Given an $n$ vertex graph whose edges are colored from one of $r$ colors $C=\set{c_1,c_2,\ldots,c_r}$, we define the Hamilton cycle color profile $hcp(G)$ to be the set of vectors  $(m_1,m_2,\ldots,m_r)\in [0,n]^r$ such that there exists a Hamilton cycle that is the concatenation of $r$ paths $P_1,P_2,\ldots,P_r$, where $P_i$ contains $m_i$ edges of color $c_i$. We study $hcp(G_{n,p})$ when the edges are randomly colored. We discuss the profile close to the threshold for the existence of a Hamilton cycle and the threshold for when $hcp(G_{n,p})=\set{(m_1,m_2,\ldots,m_r)\in [0,n]^r:m_1+m_2+\cdots+m_r=n}$.
\end{abstract}

\section{Introduction}
We are given an $n$-vertex graph where each edge is colored from a set $C=\set{c_1,c_2,\ldots,c_r}$. The Hamilton cycle color profile $hcp(G)$ is defined to be the set of vectors  $\bm \in \bM=\set{\bm\in[0,n]^r: m_1+\cdots+m_r=n}$ such that there exists a Hamilton cycle $H$ such that $H$ is the concatenation of $r$ paths $P_1,P_2,\ldots,P_r$, where $P_i$ contains $m_i$ edges of color $c_i$. 

Let $\a_1,\a_2,\ldots,\a_r$ be positive constants that sum to one and $\a$ denote $(\a_1,\a_2,\ldots,\a_r)$. Let $\gnpa$ denote the random graph $G_{n,p}$ where each edge $e$ is independently given a random color $c(e)\in C=\set{c_1,c_2,\ldots,c_r}$ where the color $c(e)$ of edge $e$ satisfies $\Pr(c (e)=c_i)=\a_i$. 

Randomly colored random graphs have been studied recently in the context of (i) rainbow matchings and Hamilton cycles, see for example \cite{BF}, \cite{CF}, \cite{FK}, \cite{FL} \cite{JW}; (ii) rainbow connection see for example \cite{DFT15}, \cite{gnprainbow}, \cite{HR}, \cite{M}, \cite{KKS}; (iii) pattern colored Hamilton cycles, see for example \cite{AF}, \cite{EFK}. This paper is closely related to Frieze \cite{FM} and Chakraborti and Hasabanis \cite{CH}, where edge-colored matchings are the topic of interest. This paper can be considered to be a contribution to the same genre. Our first theorem considers $G_{n,p}$ where $p$ is close to the Hamiltonicity threshold. For convenience, we denote the set $\set{1,2,\ldots,k}$ by $[k]$.

\begin{theorem}\label{th1}
Fix $r \ge 2$ and positive real numbers $\b,\a_1,\a_2,\ldots,\a_r$ where $\sum_{i=1}^r \a_i = 1$. If $p \ge \frac{\log n+r\log\log n+\om}{n}$ where $\om = \om(n) \rightarrow \infty$ as $n \rightarrow \infty$, then w.h.p. $hcp(\gnpa)\supseteq \bM_\b=\set{\bm\in \bM:m_i\geq \b n,i\in[r]}$. 
\end{theorem}

We will, for convenience, assume that $\om=o(\log\log n)$ and note that this also implies the theorem for larger $\om$. We next discuss why the factor $r$ in the definition of $p$ cannot be replaced by anything smaller in Theorem \ref{th1}. The importance of the factor $r$ lies in the fact that it implies that the minimum degree is at least $r+1$ w.h.p. and if we replace $\om=o(\log\log n)$ by $-\om$ then w.h.p. there will be at least $e^{\om}/2$ vertices of degree $r$. In which case there will w.h.p. be $\a_1 \a_2 \cdots \a_r e^{\om}/4$ vertices of degree $r$, all of whose incident edges have a distinct color. Thus, it is impossible to have a Hamilton cycle made from the concatenation of $r$ monochromatic paths.

Our next theorem considers when to expect $G_{n,p}$ to have a full Hamilton cycle color profile. For brevity, let $\a_{\min}=\min\set{\a_1,\ldots,\a_r}$. 

\begin{theorem}\label{th2}
Suppose that $r,\a_1,\a_2,\ldots,\a_r$ are as in Theorem \ref{th1} and that $p \ge \frac{\log n+\log\log n+\om}{\a_{\min} n}$, where $\om = \om(n) \rightarrow \infty$ as $n \rightarrow \infty$. Then, w.h.p. $hcp(\gnpa)=\bM$.
\end{theorem}

If $p\leq \frac{\log n+\log\log n-\om}{\a_{\min} n}$, then w.h.p. $hcp(\gnpa)\neq \bM$; indeed, the subgraph of $G_{n,p}$ induced by the edges of color~1 has a vertex of degree one, assuming that $\a_{\min}=\a_1$.

We finally consider directed versions of the above two theorems. Let $\dnpa$ denote the random digraph in which each edge of the complete digraph $\vec{K}_{n,p}$ occurs with probability $p$ and is randomly colored as above. We use the coupling argument of McDiarmid \cite{mcd} to prove the following couple of theorems.

\begin{theorem}\label{th3}
Suppose that $r, \b, \a_1,\a_2,\ldots,\a_r$ are as in Theorem \ref{th1} and that $p \ge \frac{\log n+r\log\log n+\om}{n}$, where $\om = \om(n) \rightarrow \infty$ as $n \rightarrow \infty$. Then, w.h.p. $hcp(\dnpa)\supseteq \bM_\b=\set{\bm\in \bM:m_i\geq \b n,i\in[r]}$. 
\end{theorem}

\begin{theorem}\label{th4}
Suppose that $r,\a_1,\a_2,\ldots,\a_r$ are as in Theorem \ref{th1} and that $p \ge \frac{\log n+\log\log n+\om}{\a_{\min} n}$, where $\om = \om(n) \rightarrow \infty$ as $n \rightarrow \infty$. Then, w.h.p. $hcp(\dnpa)=\bM$.
\end{theorem}

Note that Theorems \ref{th3} and \ref{th4} probably carry an extra $\frac{\log\log n}{n}$ in the values of $p$. This is inherent in the use of McDiarmid's argument.

\section{Preliminaries}
Throughout the paper, for clarity of presentation, we systematically omit the floor and ceiling signs when they are not crucial. This paper is organized in the following way. We start with a few standard properties of random graphs in the current section, which will be useful to prove our main results. We prove Theorems \ref{th1} and \ref{th2} in the next two sections, and prove Theorems \ref{th3} and \ref{th4} in Section \ref{pth34}. We defer the proofs of some structural lemmas for random graphs to Section \ref{fin}.

For convenience, we fix the number of colors, denoted by $r$, throughout the paper. Everywhere we will assume $n$ to be sufficiently large to support our arguments. In the following, we distinguish between events of two kinds. Those that do not depend on \bm\ and we show that they occur with probability $1-o(1)$, i.e., w.h.p. Those events that do depend on \bm\ where we need to prove that they occur with probability $1-o(n^{-r})$ in order to use the union bound on the `bad' events over all choices of $\bm\in\bM$ (note that $|\bM| = \Theta(n^r)$). We say that such events occur {\em w.v.h.p.}

The following lemma will be used in the proof of both Theorems \ref{th1} and \ref{th2}.

\begin{lemma}\label{lem0}
Suppose that $p=\frac{(c+o(1))\log n}{n}$ where $c\geq 1$ is constant. Then the following properties hold in $G_{n,p}$:
\begin{enumerate}[{\bf B1}]
\item Suppose that $S\subseteq [n]$ and $|S|=\Omega(n)$. For a vertex $v\in [n]$, we let $d_S(v)$ denote the number of neighbors of $v$ in $S$. Then, $|B(p,S)|\leq n^{1-c|S|/4n}$  w.v.h.p., where $B(p,S)=\set{v\in [n]:d_S(v)\leq \frac{c|S|\log n}{20n}}$.\label{B1}
\item Let $\SMA=B(p,[n])$. Then w.h.p., $v,w\in \SMA$ implies that $\dist(v,w)\geq 3$ in $G_{n,p}$. (Here, \dist\ refers to graph distance.)
\item Fix $S$ as in B\ref{B1}. Then w.v.h.p., every $v\in [n]$ is within distance 10 of at most $\frac{10rn}{c|S|}$ vertices in $B(p,S)$.
\item If $p = \frac{\log n+r\log\log n+\om}{n}$ with $\om = \om(n) \rightarrow \infty$ as $n \rightarrow \infty$, then $G_{n,p}$ has minimum degree at least $r+1$ w.h.p. 
\item W.v.h.p., there exists an edge between $S_1$ and $S_2$ for every $S_1,S_2\subseteq [n]$ such that $|S_1|,|S_2|\geq \frac{n(\log\log n)^2}{\log n}$ and $S_1\cap S_2=\emptyset$.
\item The maximum degree in $G_{n,p}$ is at most $5c\log n$ w.h.p.
\item $G_{n,p}$ does not contain a copy of $K_{2,3}$.
\end{enumerate}
\end{lemma}

This lemma is proved in Section \ref{plem0}.

\section{Proof of Theorem \ref{th1}}\label{pth1}
Fix a vector $\bm\in \bM_\b$ and let $\m_i=m_i/n$ for $i\in [r]$ and let $\m_{\min}=\min\set{\m_i}$. Partition the vertex set $[n]$ into $V_1,V_2,\dots, V_r$, where $V_1$ contains the first $m_1$ elements (i.e., $V_1 = [m_1]$), $V_2$ contains the next $m_2$ elements, and so on. 

We let $p_1=\frac{\log n+r\log\log n+\om/2}{n}$ and then let $p_2,p_3$ satisfy $1-p=(1-p_1)(1-p_2)(1-p_3)$ so that $p_2=p_3\approx \om/4n$. Let $d(v)$ denote the degree of $v$ in $G_{n,p_1}$ and let $d_i(v)=|\set{u \in V_i : uv \text{ has color $i$ in $G_{n,p_1}$}}|$, for $i\in [r]$. Define the following sets:
\begin{align}
A_{\bm}&=\set{v:\exists i\in[r], d_i(v)\leq \frac{\m_i\a_i \log n}{25}}.\label{defAm}\\
B&=\set{v:d(v)\leq \frac{50r^2}{\b\a_{\min}}}\label{defB}.
\end{align}
Note that $B$ is a subset of $A_\bm$.
\begin{lemma}\label{nlem}\ \\
(a) For every $\bm \in \bM_\b$, w.v.h.p, $|A_{\bm}|\leq rn^{1-\a_{\min}\m_{\min}/4}$. Thus, w.h.p. simultaneously, for all $\bm \in \bM_\b$,
\beq{smallAm}{
|A_{\bm}|\leq rn^{1-\a_{\min}\m_{\min}/4}.
}
(b) For every $\bm \in \bM_\b$, w.v.h.p. every $v \in [n]$ is within distance 10 of at most $\frac{10r^2}{\a_{\min}\m_{\min}}$ vertices of $A_\bm$. Thus, w.h.p. simultaneously, for all $\bm \in \bM_\b$, every $v \in [n]$ is within distance 10 of at most $\frac{10r^2}{\a_{\min}\m_{\min}}$ vertices of $A_\bm$.

(c) The following is w.h.p. true simultaneously for all choices of $\bm \in \bM_\b$: every pair of vertices $u\in A_\bm$ and $w\in B$ are at distance at least three in $G_{n,p_1}$.
\end{lemma}

Parts (a) and (b) of this lemma are straightforward corollaries of Properties {\bf B1} and {\bf B3} respectively. Proving Part (c) is more subtle and is done in Section \ref{yy}.

In some sense, the vertices $v$ in the set $A_\bm$ are dangerous (and we need to be careful how we place them in the Hamilton cycle). We do this by first finding vertex-disjoint paths of length two with the vertices in $A_\bm$ as the middle vertex, and then later, we make sure to include those paths in the Hamilton cycle. 

We now give an outline of the way we will construct a Hamilton cycle in several steps. Later we will elaborate on why these steps are valid, assuming the high probability events stated in Lemmas \ref{lem0} and \ref{nlem}.

\begin{enumerate}[{\bf Step 1}]
\item We first argue that for each $v\in A_\bm$, we can choose a path $Q_v=(w_1,v,w_2)$ where $w_1,w_2\notin A_\bm$ and both edges of $Q_v$ have the same color, $c_j$, say. Let $\cQ=\set{Q_v : v \in A_\bm}$ and let $\cQ_i\subseteq \cQ$ be the set of paths contained in $V_i, i\in [r]$. The paths $Q_v,v\in A_\bm$ can be chosen to be vertex disjoint. Next, we move $v,w_1,w_2$ to $V_j$ and move three vertices in $V_j\setminus (A_\bm\cup N(v))$ to the sets originally containing $v,w_1,w_2$, in order to keep the sizes of the $V_i$'s unchanged. 


Following this step, for each $i\in [r]$, let $G_i$ denote the subgraph of $G_{n,p_1}$ with vertex set $V_i$ and edges of color $i$. 

\item For each $i\in [r]$, execute a restricted {\em rotation-extension} algorithm where at all times, we ensure that for all $Q\in\cQ_i$, the current path either contains $Q$ or is vertex disjoint from $Q$. In this way, create a Hamilton path $H_i$ through $V_i$ for $i\in [r]$. 

\item Connect the Hamilton paths constructed in Step 2 into a Hamilton cycle.
\end{enumerate}

\subsection{Validation of Step 1}\label{validation:step1}
Property {\bf B4} and the pigeonhole principle imply that for each $v\in A_\bm$, we can choose two neighbors $w_1,w_2$ such that the edges $vw_1,vw_2$ have the same color. If $v\in B$, then arbitrarily pick two neighbors $w_1,w_2$ such that the edges $vw_1,vw_2$ have the same color; Lemma \ref{nlem}(c) implies that $w_1,w_2 \notin A_\bm$. If $v_1,v_2\in B$ then Property {\bf B2} ensures that the corresponding paths $Q_{v_1},Q_{v_2}$ are vertex disjoint. 

If $v\in A_\bm\setminus B$, then $d(v)>\frac{50r^2}{\b\a_{\min}}$ and $v$ has at most $\frac{10r^2}{\a_{\min}\m_{\min}}$ neighbors in $A_\bm$, from Lemma \ref{nlem}(b). Moreover, by Lemma \ref{nlem}(b) and Property {\bf B7}, the vertex $v$ has at most $\frac{20r^2}{\a_{\min}\m_{\min}}$ neighbors $w$ such that $w$ has at least one neighbor in $A_\bm\setminus \set{v}$. Thus, for each $v\in A_\bm\setminus B$, we have at least $\frac{50r^2}{\b\a_{\min}}-\frac{30r^2}{\a_{\min}\m_{\min}} \ge \frac{20r^2}{\b \a_{\min}}$ choices of neighbors which are neither in $A_\bm$ nor in the neighorhood of some vertex in $A_\bm\setminus \set{v}$. As a consequence, in a greedy manner, we can construct a path $Q_v=(w_1,v,w_2)$ for each $v\in A_\bm\setminus B$ such that $w_1,w_2 \notin A_\bm$, the edges $vw_1,vw_2$ have the same color, and each $Q_v$ is vertex disjoint. Note also that if $v_1\in A_\bm$ and $v_2\in B$ then we can use Lemma \ref{nlem}(c) to argue that $Q_{v_1},Q_{v_2}$ are vertex disjoint.


\subsection{Validation of Step 2} \label{validation_of_step2}
Call a neighbor $w$ of a vertex $v$ {\em bad} if $(\{w\}\cup N(w))\cap A_\bm\neq \emptyset$. In Step 1, only  the bad neighbors of $v \notin A_\bm$ can reduce the $V_i$-neighborhood of $v$. Lemma \ref{nlem}(b) implies that for each $v \notin A_\bm$, the number of neighbors of $v$ in $G_i$ can drop by at most $\frac{30r^2}{\a_{\min}\m_{\min}}$. Thus, the vertices of $G_i$, not in $A_\bm$, have degree at least $\frac{\m_{\min}\a_{\min}\log n}{25}-\frac{30r^2}{\a_{\min}\m_{\min}}\geq \frac{\m_{\min}\a_{\min} \log n}{26}$.

\subsubsection{Expansion properties}\label{exprop}
We need to show that each $G_i$ has certain expansion properties. We have the following properties of $G_i \subseteq G_{n,p_1}$, which will be verified in Section \ref{fin}. For a set $S\subseteq V_i$, let $N_i(S)=\set{w\in V_i \setminus S:\exists v\in S\ s.t.\ vw\in E(G_i)}$.

\begin{lemma} \label{lem2}
The following properties hold for all $i\in [r]$ w.v.h.p. 
\begin{enumerate}[(a)]
\item For every set $S\subseteq V_i\setminus A_\bm$ with $|S|\leq n/\log^4n$, we have that $|N_i(S)|\geq |S|\m_{\min} \a_{\min} \log n/1000$.
\item For every set $S\subseteq V_i\setminus A_\bm$ with $|S|\leq \m_{\min}^2 \a_{\min}^2 n/10^7$, we have that $|N_i(S)|\geq 3|S|$.
\item The graph induced by color $i$ on the vertex set $V_i \setminus A_\bm$ is connected.
\end{enumerate}
\end{lemma}

This lemma is proved in Section \ref{xx}.

\subsubsection{Step 2: Constructing Hamilton paths in $G_i$}\label{step4}
We now validate Step 2 in a stronger sense. More precisely, we prove that there are many Hamilton paths in each $G_i$. This will later be used in gluing them together to obtain a Hamilton cycle of $G$. Let 
\[
n_0=\frac{\m_{\min}^2 \a_{\min}^2 n}{10^7}.
\]

\begin{lemma} \label{lem3}
W.h.p. simultaneously, for all $\bm\in\bM_\b$, the following two events occur in $G_{n,p_1}\cup G_{n,p_2}$: Each $G_i$ has at least $n_0$ vertices $v$ for which there are at least $n_0$ Hamilton paths with one end point $v$ such that the other end points are pairwise distinct.
\end{lemma}

\begin{proof}
Although by now extension-rotation is a standard procedure for attacking Hamilton cycle problems, we briefly describe it here. Given a path $P=(x_1,x_2,\ldots,x_k)$ an {\em extension} is simply the creation of a new path $P+(x_k,y)$ where $x_ky$ is an edge and $y\notin V(P)$. If $1<i\leq k-2$ and $x_kx_i$ is an edge then we create a new path $(x_1,x_2,\ldots,x_i,x_k,x_{k-1}\ldots,x_{i+1})$ of the same length as $P$ by a {\em rotation} with {\em fixed endpoint} $x_1$. We let $END=END(P,x_1)$ denote the set of vertices that can be the endpoint of a path created by a sequence of rotations.

We modify the above constructions on $G_i$ by adding the restriction that for each $Q \in \cQ_i$, the paths generated either contain $Q$ or are vertex disjoint from $Q$. We can do this by always adding or deleting both edges of such a path in any change. Any rotation that would result in deleting one edge of such a path is neglected. Under the assumption that $P$ is a longest path so that there are no extensions, P\'osa \cite{Posa} proved that $|N(END)|<2|END|$ and then accounting crudely for the interiors of the paths of $\cQ$ we see that the endpoint sets satisfy
\beq{Nend}{
|N(END)|\leq 2|END| + \min\set{2|\cQ_i|, \frac{10r^2}{\m_{\min} \a_{\min}}|END|}.
}
Indeed, suppose that $v\in END$ and $w\in N(v) \subseteq N(END)$ and that $x,y$ are the neighbors of $w$ in $P$. Consider the path $P'$ with endpoint $v$ obtained by rotations. If either of the edges $wx$ or $wy$ are deleted in this sequence, then at least one of $x,y$ is in $END$, accounting for the $2|END|$ term as usual. So, if neither $x$ nor $y$ are in $END$, then $x,w,y$ is a subpath of $P'$, and we cannot rotate using $vw$ because it would destroy some $Q\in\cQ$. This can happen at most $|\cQ_i|$ times accounting for the $2|\cQ_i|$ in \eqref{Nend}. The bound $\frac{10r^2}{\m_{\min} \a_{\min}}|END|$ arises from Lemma \ref{nlem}(b), because at least one of $x,w,y$ must be in $A_\bm$ for the blocking of a rotation.

Since $|\cQ_i| \le |A_\bm| \ll n/\log^4n$ for each $\bm \in \bM_\b$ (by \eqref{smallAm}), we can deduce from Lemma \ref{lem2} that w.h.p. for each $\bm \in \bM_\b$, the endpoint sets are of size at least $n_0$. We show next that with the use of $G_{n,p_2}$, we can prove that each $G_i$ has a Hamilton path w.h.p. More precisely, suppose that $E(G_{n,p_2})=F=\set{f_1,f_2,\ldots,f_\s}$ where w.h.p. $\s\geq \om n/5$. Partition $F$ into $r$ sets $F_1,F_2,\ldots,F_r$ of almost equal size.

Condition on the high probability events in Lemmas \ref{lem0}, \ref{nlem}, and \ref{lem2}. Now given a path $P$ of length $\ell<m_i-1$ in $G_i$, we make a series of rotations with one endpoint fixed until either the endpoint set $END$ reaches $n_0$ in size, or we generate a path that can be extended. Assume the former. Then for each $v\in END$, there is a path $P_v$ of length $\ell$ and one endpoint being $v$. We then try to find a longer path by doing rotations and extensions with $v$ as the fixed endpoint. We do this for all $v\in END$. If we never extend a path, then we terminate with a set $END$ of $n_0$ vertices, and for each $v\in END$, a set of $n_0$ paths with distinct endpoints $END_v$. Observe next that adding an edge $f=vw$ where $w\in END_v$ will enable us to create a path of length $\ell+1$. This is because adding $f$ creates a cycle $C$ of length $\ell+1$. Now $G_i$ is connected. This follows from Lemma \ref{lem2}(c) and the fact that each vertex $v\in A_\bm\cap V_i$ is contained in a path $x,v,y$ where $x,y$ are in $V_i\setminus A_\bm$. We can find a path of length $\ell+1$ by adding an edge $g_1=ww_1$ and deleting an edge $g_2=ww_2$ where $g_2\in E(C)$ and $w_1\notin V(C)$. The edge $f$ is referred to as a {\em booster}.

If we go through the edges of $F_i$ one by one, we see that each edge has probability at least $\g=\a_{\min}n_0^2/3n^2$ of being a booster. This bound holds given the previous edges examined. Thus the probability we fail to obtain a Hamilton path in each $G_i$ is bounded by the probability that the binomial random variable $B(\s/r,\g)<n$, which is bounded by $e^{-\Omega(n)}$. After a simple application of union bound, this shows that w.h.p. for each $\bm \in \bM_\b$, we can find Hamilton paths in each $G_i$ as desired. 
\end{proof}
\subsection{Step 3: Connecting the Hamilton paths together}\label{step5}
In the final step, our goal is to show that w.h.p. we can choose Hamilton paths $P_i$ of $G_i$ with endpoints $x_i$ and $y_i$ for $i=1,2,\ldots,r$, such that for each $i$, the edge $y_ix_{i+1}$ exists and is colored with $c_{i+1}$. We begin by choosing $n_0$ hamilton paths in $G_1$ all with vertex $x_1'$, say as one endpoint.

Assume inductively, that we have chosen $P_1,P_2,\ldots,P_{i-1}$ plus $n_0$ Hamilton paths $Q_1,Q_2,\ldots,Q_{n_0}$ of $G_i$, all with endpoint $x_i$ (or $x_1'$ if $i=1$). Now choose a set $END_{i+1}$ of size $n_0$ such that each $v\in END_{i+1}$ is the endpoint of $n_0$ Hamilton paths of $G_{i+1}$ with distinct endpoints. We now use the edges of $G_{n,p_3}$ to find a vertex $x_{i+1}\in END_{i+1}$ such that there is an edge $yx_{i+1}$ of color $c_{i+1}$, where $y\neq x_i$ is an endpoint of one of the paths $Q_1,Q_2,\ldots,Q_{n_0}$. Similarly to the last time, suppose that $E(G_{n,p_3})=F_0$ where w.h.p. $|F_0|\ge \om n/5$. As we go through the edges of $F_0$, we see that we find such an edge with probability at least $\g$. It follows that for each $\bm \in \bM_\b$, the probability that we fail to find the required edge after $\log^2n$ steps is at most $(1-\g)^{\log^2n}$. Repeating this argument $r$ times we see that w.h.p. for each $\bm \in \bM_\b$, there are $n_0$ Hamilton paths of $G$ made up of correctly colored paths of length $m_1-1,m_2,\ldots,m_{r-1}$ plus one of $n_0$ Hamilton paths $H_1,H_2,\ldots,H_{n_0}$ of $G_r$, all with $x_r$ as an endpoint. 

We now do rotations in $G_1$, starting with $P_1$ and keeping the endpoint $y_1$ fixed and generate $n_0$ paths $J_1,J_2,\ldots,J_{n_0}$. We then search for an edge $y_rx_1$ of color $c_1$ such that $y_r$ is an endpoint of an $H_k$ and $x_1$ is an endpoint of a $J_l$. We can find one w.h.p. for each $\bm \in \bM_\b$ by examining $\log^2n$ edges of $F_0$, and we are done with the proof of Theorem \ref{th1}.

\section{Proof of Theorem \ref{th2}}\label{pth2} 
 Fix a vector $\bm\in\bM$. Clearly, there exists $j\in [r]$ such that $m_j \ge n/r$; without loss of generality, assume that $m_r \ge n/r$ (because any cyclic shift of the coordinates in $\bm$ is precisely a cyclic switching of the colors). We focus initially on the first $r-1$ colors and construct paths $P_j$ for each $j\in [r-1]$ so that we can construct the remaining long path $P_r$ using a somewhat similar strategy as before to glue all the paths together. 

We partition the vertex set $[n]$ into $V_1,\ldots, V_{r-1}, V_{r,1}, V_{r,2}, V_{r,3}$, where these sets are inductively defined as follows. For $j\in [r-1]$, the set $V_j$ consists of the interval of $m_j + \frac{n}{10r^2}$ vertices starting from $1 + \sum_{i\le j-1} |V_i|$. Each of $V_{r,1}$ and $V_{r,2}$ consists of the interval of $\frac{m_r}{3} + \frac{n}{20r^2}$ vertices starting from $1 + \sum_{i\le r-1} |V_i|$ and $1 + |V_{r,1}| + \sum_{i\le r-1} |V_i|$ respectively. Then, finally the set $V_{r,3}$ contains the remaining $\frac{m_r}{3} - \frac{n}{10r}$ vertices forming an interval ending at $n$. For each $j\in [r-1]$, we use the vertices in $V_j$ to construct a path of color $c_j$. We use $\cup_{i=1}^3 V_{r,i}$ for the last color $c_r$, in addition to some vertices transferred from outside this set. Let $\m_j$ be such that $|V_j| = \m_j n$ for $j\in [r-1]$. Let $\m = n^{-1}\brac{\frac{m_r}{3} - \frac{n}{10r}}$. 

We let $p_1=\frac{\log n+\log\log n+\om/2}{\a_{\min}n}$ and then let $p_2,p_3$ satisfy $1-p=(1-p_1)(1-p_2)(1-p_3)$ so that $p_2 = p_3 \approx \om/4\a_{\min}n$. For each $v\in[n]$ and $i=1,2,3$, let $d_{r,i}(v)=|\set{u \in V_{r,i} : uv \text{ has color } c_r \text{ in } G_{n,p_1}}|$. For each $v\in[n]$ and $j\in [r-1]$, let $d_j(v)=|\set{u \in V_j : uv \text{ has color } c_j \text{ in } G_{n,p_1}}|$. Finally, for each $v\in[n]$, let $d_r(v)=|\set{u \in V : uv \text{ has color } c_r \text{ in } G_{n,p_1}}|$ (note that this is different from $d_j(v), j\neq r$; the notation $d_r(v)$ denotes the number of $c_r$-colored edges incident to $v$ in $G_{n,p_1}$).

We now outline how we will construct a Hamilton cycle in several steps. Let now 
\begin{align*}
A_\bm&=\set{v:\exists j\in[r-1]: d_j(v)\leq \frac{\m_i\a_i \log n}{25} \;\;\; \text{or} \;\;\; \exists 1\leq i\leq 3: d_{r,i}(v)\leq \frac{\m\a_r \log n}{25}}.\\
B&=\set{v:d_r(v)\leq \frac{500r^4}{\a_{\min}}}.
\end{align*}

\begin{lemma}\label{addlemma}\ \\\vspace{-.5in}
\begin{enumerate}[(a)]
\item W.h.p. simultaneously, for all choices of \bm, every pair of vertices $u\in A_\bm$ and $v\in B$ are at distance at least three.
\item W.h.p., $d_r(v)\geq 2$ for all $v\in [n]$.
\end{enumerate}
\end{lemma}

This lemma is proved in Section \ref{xxx}.

We next describe the steps of our construction.
\begin{enumerate}[{\bf Step 1}]
\item For each $v\in A_{\bm}$, choose two neighbors $w_1,w_2\notin A_\bm$ of $v$ such that $vw_1$ and $vw_2$ have the color $c_r$ and let $Q_v$ be the path $w_1vw_2$. Then, move $v$, $w_1$, and $w_2$ to $V_{r,3}$. We will show in Section \ref{step1:Qv} that we can choose the pairs $w_1,w_2$ such that the paths in $\cQ=\set{Q_v}$  are vertex disjoint. 

After this step for $j\in [r-1]$, denote the new $V_j$ by $V'_j$, and for each $i=1,2,3$, denote the modified $V_{r,i}$ by $V'_{r,i}$.

\item Construct a path $P$ of the form $P_1 P_2 \ldots P_{r-1}$, where $P_j$ is a path using only the vertices of $V'_j$ as internal vertices, and with edges of color $c_j$. The path $P_j$ has length $m_j$ for $j \in [r-1]$. For all vertices in $\cup_{j=1}^{r-1} V'_j$ that are not used in $P$, place them into $V'_{r,3}$. 

Let $G_j$ denote the subgraph induced by the edges of color $c_j$ in $V'_j$, for $j\in [r-1]$.
Let $G_{r,i}$ denote the subgraph induced by the edges of color $c_r$ in $V'_{r,i}$, for $i=1,2,3$. With similar arguments as in Sections \ref{validation_of_step2} and \ref{step1:Qv}, for $i=1,2$, the graphs $G_{r,i}$ have minimum degree at least $\frac{\m \a_r \log n}{25}-360r^4\geq \frac{\m \a_{\min} \log n}{26}$ by construction. The same is true for the degrees in $G_{r,3}$, except for the vertices of $A_\bm$.

\item For each $i=1,2$, execute the rotation-extension algorithm on $G_{r,i}$ to find an almost Hamilton path $P_{r,i}$ and connect one end of $P_{r,1}$ to $P_1$ and one end of $P_{r,2}$ to $P_{r-1}$ so that there are linearly many choices for the other ends of $P_{r,1}$ and $P_{r,2}$. Move all the unused vertices from $V'_{r,1}\cup V'_{r,2}$ to $V'_{r,3}$ and again perform the restricted rotation-extension algorithm on $G_{r,3}$ to ensure that for all $Q\in\cQ$, the current path either contains $Q$ or is vertex disjoint from $Q$ and build $P_{r,3}$ using the remaining vertices. Finally place $P_{r,3}$ in between $P_{r,1}$ and $P_{r,2}$ to complete the Hamilton cycle. 
\end{enumerate}

\subsection{Step 1: Construction of disjoint $Q_v$'s}\label{step1:Qv}
To show that the pairs $w_1,w_2$ can be chosen so that the paths in $\cQ=\set{Q_v}$  are vertex disjoint, we use a similar argument from the validation of Step 1 in the proof of Theorem \ref{th1}. Lemma \ref{addlemma}(b) implies that for each $v\in A_\bm$, we can choose two neighbors $w_1,w_2$ such that the edges $vw_1, vw_2$ have the color $c_r$. 
If $v\in B$, then arbitrarily pick two such neighbors $w_1,w_2$; Lemma \ref{addlemma}(a) implies that the chosen vertices neither are in $A_\bm$ nor have any neighbors in $A_\bm$. 

If $v\in A_\bm \setminus B$, then $d_r(v)> \frac{500r^4}{\a_{\min}}$. Next, we apply {\bf B3} to $c=\frac{\a_i}{\a_{min}}, S = V_j$ for each $1 \le j \le r-1$ and also to $c=\frac{\a_r}{\a_{min}}, S = V_{r,i}$ for each $1 \le i \le 3$; and then sum the results, to obtain an upper bound on the number of vertices in $A_\bm$ which are at distance at most 10 from $v$. Applying {\bf B3} for a given $1\le j\le r-1$ gives a bound of $\frac{10rn\a_{min}}{\a_i |V_j|} \le 100r^3$ (since $|V_j| \ge n/10r^2$), and similarly for a given $1\le i\le 3$, it gives a bound of $\frac{10rn\a_{min}}{\a_r |V_{r,i}|} \le 40r^2$ (since $|V_{r,i}| \ge n/4r$).
By summing these, we have that the number of vertices in $A_\bm$, which are at distance at most 10 from $v$, is at most $120r^4$. Thus, a similar argument as in Section \ref{validation:step1} using Property {\bf B7} shows that the vertex $v$ has at most $240r^4$ neighbors $w$ such that $w$ has at least one neighbor in $A_\bm\setminus \set{v}$.
Thus, for each $v\in A_\bm \setminus B$, we have at least $\frac{500r^4}{\a_{\min}}- 360r^4 \ge 140r^4$ choices of neighbors (to pick $w_1,w_2$ from) that are disjoint from already chosen path endpoints. 

\subsection{Step 2: Construction of paths $P_1,P_2,\ldots,P_{r-1}$}
To obtain $P_1,P_2,\ldots,P_{r-1}$ we use the following lemma. (See Ben-Eliezer, Krivelevich, and Sudakov \cite{BKS}.) 

\begin{lemma} \label{lemBKS}
Let $G$ be a connected graph with at least $N$ vertices such that for every pair of disjoint sets $S$ and $T$ with $|S|=|T|=M$, there is an edge joining $S$ and $T$. Then for every $v\in V(G)$, there is a path of length $N-2M$ with one endpoint $v$.
\end{lemma}

We need the following lemma which enables us to apply Lemma \ref{lemBKS} on the graphs $G_i$ for $i \in [r-1]$.

\begin{lemma}\label{BKS}
W.h.p. simultaneously, for all choices of \bm, for each $i\in [r-1]$, we have the following: 
\begin{enumerate}[(a)]
\item $G_i$ is connected and 
\item There is an edge in $G_i$ between every pair of disjoint sets $S$ and $T$ with $|S| = |T| = n_1 = \frac{n(\log \log n)^2}{\log n}$.
\end{enumerate}
\end{lemma}

This will be proved in Section \ref{pbks}.

We condition on the high probability events in the above lemmas. We assume that $m_j \ge 1$ for all $j\in [r]$ because otherwise, we are just dealing with fewer colors. Fix a starting vertex $v_1 \in V'_1$. It follows from Lemmas \ref{lemBKS} and \ref{BKS} that there is a path $P_1$ of length $m_1$ starting at $v_1$ and using only the vertices in $V'_1$, all of whose edges have color $c_1$ (we use Lemma \ref{lemBKS} with $N=m_1+\frac{n}{10r^2}-|A_\bm|\sim m_1+\frac{n}{10r^2}$ and $M=n_1$). Suppose then that we have constructed paths $P_1,P_2,\ldots,P_k,k<r-1$ where $P_{j-1},P_{j}$ share an endpoint and the edges of $P_j$ are colored $c_j$ for $1\leq j\leq k$. (We take $P_0$ to be an endpoint of $P_1$.) Let $u_k$ denote the endpoint of $P_k$ that is not in $P_{k-1}$ and $v_{k+1}$ be a $c_{k+1}$-neighbor of $u_k$ in $V'_{k+1}$ (such a neighbor exists because of the fact that $V'_k$ only contains vertices outside of $A_\bm$). Then it follows from Lemmas \ref{lemBKS} and \ref{BKS} that there is a path $P_{k+1}$ of length $m_{k+1}$ starting at $u_{k}$ and using only the vertices in $V'_{k+1} \cup \{u_k\}$, all of whose edges have color $c_{k+1}$. We end the path $P_{r-1}$ with a vertex $u_{r-1} \in V'_{r-1}$. To summarise, we have constructed a path, the concatenation of $P_1, P_2, \ldots, P_{r-1}$, starting from $v_1$ and ending at $u_{r-1}$ such that the edges of $P_j$ are colored with $c_j$. 

\subsection{Step 3: Construction of $P_r$ and the Hamilton cycle}\label{caveat}
Our goal in this section is to construct a path $P_r$ between the vertices $v_1$ and $u_{r-1}$ using edges of color $c_r$. And using all of the unused vertices outside of $\cup_{j=1}^{r-1} P_j$ as the internal vertices. Step 3 can be validated in a similar way as was done in Sections \ref{step4} and \ref{step5}, and note that Lemma \ref{lem2} continues to hold for the graphs $G_{r,i}$, $i\in [3]$ (we need a minor modification as mentioned later). There is one caveat in that we want the path $P_r$ to start with the fixed vertex $v_1$ and end with the fixed vertex $u_{r-1}$. In contrast, we previously had linearly many options for starting or ending vertices. Thus, for $i=1,2$, we aim to first construct a family $\cP_{r,i}$ of linearly many paths using almost all vertices in $V'_{r,i}$ such that the paths in $\cP_{r,1}$ start with $v_1$, the paths in $\cP_{r,2}$ start with $u_{r-1}$, and the other endpoints are pairwise disjoint. This will lead to a situation similar to the proof of Theorem \ref{th1}, where we can finish by constructing a final path $P_{r,3}$ using the vertices in $V'_{r,3}$ and all unused vertices from $\cup_{i\in [2]} V'_{r,i}$ that connect some $P_{r,1}\in \cP_{r,1}$ and $P_{r,2}\in \cP_{r,2}$. 

To this end, observe that the vertex $v_1$ has at least $\ell_0=\th\log n$ neighbors $v\in V'_{r,1}$ such that $vv_1$ is an edge of color $c_r$. Indeed, since $v_1\notin A_\bm$, we know that $v_1$ has at least $\th\log n$ neighbors in $V'_{r,1}$, where $\th=\frac{\m\a_r}{26}$ (by an argument used in the description of Step 1). Fix such a set $N_1 \subseteq V'_{r,1}$ of size exactly $\th\log n$ such that for all $v\in N_1$, there is an edge $vv_1$ of color $c_r$. By a similar argument, we fix another set $N_2 \subseteq V'_{r,2}$ of size exactly $\th\log n$ such that for all $u\in N_2$, there is an edge $uu_{r-1}$ of color $c_r$. 

The subgraph $G'_{r,1}$ obtained from $G_{r,1}$ by deleting the vertices in $N_1$ satisfies the expansion properties of Section \ref{exprop}. (the proof of Lemma \ref{lem2} is still valid; note that the removal of the set $N_1$ cannot decrease the minimum degree of $G'_{r,1}$ by more than 2, because otherwise there would be a copy of $K_{2,3}$ in $G_{n,p_1}$ contradicting {\bf B7}.) Denote by $G'$ the subgraph of $G_{n,p_1} \cup G_{n,p_2}$ induced by the vertex set $V(G'_{r,1})$. Using the same arguments as in Sections \ref{step4} and \ref{step5}, we can find a set $END$ of $n_0$ vertices $v$ for which there are at least $n_0$ Hamilton paths in $G'$ with one end point $v$ and otherwise distinct endpoints. The probability there is no edge of color $c_r$ from $N_1$ to $END$ is then at most $\brac{1-p_2}^{n_0\th\log n}\leq n^{-\th\omega/5\a_{\min}}=o(n^{-r})$. Thus, there is such an edge $vv'$ with $v\in N_1$ and $v'\in END$. This completes the construction of the family $\cP_{r,1}$ of $n_0$ paths starting with the vertices $\set{v_1,v,v'}$, ending at distinct vertices (denote this set of vertices by $Z_1$), and using every vertex in the set $V(G'_{r,1}) \setminus N_1$. 

A similar argument provides us with a vertex $u\in N_2$ and a family $\cP_{r,2}$ of $n_0$ paths starting with the vertices $\set{u_{r-1},u}$, ending at distinct vertices (denote this set of vertices by $Z_2$), and using every vertex in the set $V(G'_{r,2}) \setminus N_2$; with a failure probability $o(n^{-r})$. At this stage move every vertex from $N_1 \setminus \set{v}$ and $N_2 \setminus \set{u}$ to $V'_{r,3}$, and denote this modified set by $V''_{r,3}$. 

Finally, Lemma \ref{lem3} holds for the subgraph of $G_{n,p}$ induced by $V''_{r,3}$ because the arguments in Section \ref{step4} remain valid. Finally, with a similar argument to that in Section \ref{step5}, we can join the endpoint of a Hamiltonian path $P_{r,3}$ on $V''_{r,3}$ to the open end of a path in $\cP_{r,1}$ and to the open end of another path in $\cP_{r,2}$; with a failure probability $o(n^{-r})$. This finishes the proof of Theorem \ref{th2}.

\section{Proof of Theorems \ref{th3} and \ref{th4}}\label{pth34}
We can consider both theorems simultaneously. Let $q=p(1-p)$ and note that $\gnqa$ satisfies the conditions of Theorems \ref{th1} and \ref{th2}. 

Let $c$ be a fixed coloring that we will use to color edges. Now let $e_i=\set{u_i,v_i},i=1,2,\ldots,N=\binom{n}{2}$ be an arbitrary ordering of the edges of $K_n$. We couple the construction of $\gnqa,q=p(1-p)$ with $\dnpastar$, a subgraph of $\dnpa$. For each $i$, we generate two independent Bernouilli random variables, $B_{u_i,v_i}$ and $B_{v_i,u_i}$, each with probability of success $p$. If exactly one of these variables has value one, we include the corresponding directed edge in $\dnpastar$ and give it the color $c(e_i)$. 

Consider the following sequence $\G_0,\G_1,\ldots,\G_N$ of random edge-colored digraphs. In $\G_i$, for $j\leq i$, we first tentatively include $(u_j,v_j)$ and $(v_j,u_j)$ independently with probability $p$ and include the corresponding edge only if exactly one is chosen. In which case give it color $c(e_j)$. For $j>i$ we include both $(u_j,v_j),(v_j,u_j)$ with probability $q$ and neither of $(u_j,v_j),(v_j,u_j)$ with probability $1-q$. 

Now $\G_0$ is distributed as $\gnqa$ and $\G_N$ is distributed as a subgraph of $\dnpa$. We argue that
\beq{MCD}{
\Pr(\G_i\in\cF)\geq \Pr(\G_{i+1}\in\cF)\text{ for }0\leq i<N.
}

Given \eqref{MCD} we see that we have Theorems \ref{th3} and \ref{th4}. So let us verify \eqref{MCD}. Following \cite{mcd}, we condition on the existence or non-existence of $(u_j,v_j)$ or $(v_j,u_j)$ for $j\neq i+1$, in both models, $\G_i,\G_{i+1}$. Let $\cC$ denote this conditioning. Then, one of (a), (b), (c) below occurs:
\begin{itemize}
\item[(a)] There is a desiredly colored Hamilton cycle (in both $\G_i,\G_{i+1}$) that does not use either of  $(u_{i+1},v_{i+1})$ or $(v_{i+1},u_{i+1})$.
\item[(b)] Not (a) and there exists a desiredly colored Hamilton cycle if at least one of $(u_{i+1},v_{i+1})$ or $(v_{i+1},u_{i+1})$ is present, or
\item[(c)] There does not exist a desiredly colored Hamilton cycle even if both of $(u_{i+1},v_{i+1})$ and  $(v_{i+1},u_{i+1})$ are present.
\end{itemize}

(a) and (c) give the same conditional probability of Hamiltonicity in $\Gamma_i,\Gamma_{i+1}$, 1 and 0 respectively. In $\Gamma_{i}$, (b) happens with probability $q$. In $\Gamma_{i+1}$, we consider two cases (i) exactly one of $(u_{i+1},v_{i+1}),(v_{i+1},u_{i+1})$ yields Hamiltonicity and in this case the conditional probability is again $q$ and (ii) either of $(u_{i+1},v_{i+1}),(v_{i+1},u_{i+1})$ yields Hamiltonicity and in this case the conditional probability is  $1-(1-p)^2-p^2=2q$. Note that we will never require that {\bf both} $(u_{i+1},v_{i+1}),(v_{i+1},u_{i+1})$ occur. In summary, we have proved that
\beq{star}{
\Pr(\dnpastar\in\cF)\leq \Pr(G^\a_{n,p}\in\cF)=o(1).
}

\section{Structural lemmas}\label{fin} 
In this section, we prove the various structural properties of random graphs used throughout this paper. We begin with the following: let $0<\g<1$ and $g=\rdown{1/\g}$ and let $W_1,W_2,\ldots,W_g$ be consecutive intervals in $[n]$ where $|W_i|=\rdown{\g n}$ for $1\leq i<g$. Let $d_{i,j}(v)$ denote the number of neighbors $w$ of vertex $v$ in $W_j$ such that $c(vw)=c_i$. Here $G=G_{n,p_1}$ with $p_1\approx \frac{c\log n}{n},c\geq 1$. Let 
\begin{align*}
A^*&=\set{v:\exists i\in[r],j\in[g-1]:d_{i,j}(v)\leq \frac{\g c\a_i\log n}{20}}.\\
B_1&=\set{v:d(v)\leq \frac{5r^2}{\g\a_{\min}}}.\\
B_2&=\set{v:d_r(v)\leq \frac{5r^2}{\g\a_{\min}}}.
\end{align*}

\begin{lemma}\label{basic}\ 
\begin{enumerate}[(a)]
\item In Section \ref{pth1} with $c=1$, $\g=\b/10$, we have that $A_\bm \subseteq A^*$ and $B_1 = B$.
\item  In Section \ref{pth2} with $c=1/\alpha_{\min}$, $\g=1/100r^2$, we have that $A_\bm\subseteq A^*$ and $B_2 = B$.
\end{enumerate}
\end{lemma}

\begin{proof}
It is clear that $B_1 = B$ in (a) and $B_2=B$ in (b).

(a) If $v \in A_\bm$, then there is some $i \in [r]$ such that $d_i(v) \le \frac{\m_i \a_i \log n}{25}$, i.e., there are at most $\frac{\m_i \a_i \log n}{25}$ many $c_i$-colored edges between $v$ and $V_i$. Recall that $V_i$ was defined so that it consists of $m_i \ge \b n$ consecutive elements from $[n]$. Hence, there are $j,k$ such that $W_{j} \cup W_{j+1} \cup \cdots \cup W_{j+k-1} \subseteq V_i$ and $V_i \setminus (W_{j} \cup W_{j+1} \cup \cdots \cup W_{j+k-1})$ has at most $\b n/5$ elements. Thus, 
$$k\g n \geq |W_{j} \cup W_{j+1} \cup \cdots \cup W_{j+k-1}| \ge m_i - \b n/5 \ge 4m_i/5.$$
Suppose, for the sake of contradiction, that $v \notin A^*$. Then, $d_{i,j+l}(v) > \frac{\g \a_i \log n}{20}$ for all $l= 0,1,\ldots,k-1$. Thus, we have the following (recall that $m_i = \m_i n$):
\begin{align*}
d_i(v) \ge \sum_{l=0}^{k-1} d_{i,j+l}(v) > \frac{k\g \a_i \log n}{20} \ge \frac{\m_i \a_i \log n}{25},
\end{align*}
giving us a contradiction. 

(b) The proof is essentially the same as part (a).
\end{proof}

\begin{lemma} \label{lemma1}\ 
\begin{enumerate}[(a)]
\item If $c=1$, $\g=\b/10$, then w.h.p. every pair of vertices $u \in A^*$ and $v \in B_1$ are at distance at least three.
\item If $c=1/\a_{\min}$, $\g=1/100r^2$, then w.h.p. every pair of vertices $u \in A^*$ and $v \in B_2$ are at distance at least three.
\end{enumerate}
\end{lemma}

\begin{proof}\ 
(a) The probability that there are vertices $u \in A^*$ and $v \in B_1$ at distance at most two can be bounded by
\mults{
\sum_{i=1}^2 n^{i+1} \bfrac{c \log n}{n}^i \sum_{k=1}^{5r^2/\g \a_{\min}} \binom{n-1-i}{k} p^k (1-p)^{n-1-i-k}  \sum_{i=1}^{r} \sum_{j=1}^{g-1} \sum_{k=1}^{\g c \a_i\log n/20} \binom{\g n}{k}(p\a_i)^{k}\brac{1-p\a_i}^{\g n -2-i-k} \\
\leq O\brac{n\times \log^2n\times \sum_k\log^kn\times n^{-1}\times \sum_{i,j,k}\bfrac{e^{1+o(1)}\g\a_i \log n}{k}^k\times n^{-\Omega(1)}}= o(1).
}

(b) This is similar.
\end{proof}

\subsection{Proof of Lemma \ref{lem0}}\label{plem0}

\begin{enumerate}[{\bf B1}]

\item Let $Z=|B(p,S)|$ and $L=\log n$ and $A=\frac{c|S|\log n}{20n}$. Then,

\begin{align}
\E\brac{\binom{Z}{L}}&\leq \binom{n}{L}\brac{\sum_{i=0}^A\binom{|S|-L}{i}p^i(1-p)^{|S|-L-i}}^L\label{B1a}\\
&\leq \binom{n}{L}\brac{2\binom{|S|}{A}p^A(1-p)^{|S|-A}}^L\nn\\
&\leq \binom{n}{L}\brac{2\brac{\frac{|S|e}{A}\cdot\frac{c\log n}{n}\cdot e^{o(1)}}^A e^{-c|S|\log n/n}}^L\nn\\
&\leq \binom{n}{L}((21e)^{\log n/20}n^{-1+o(1)})^{Lc|S|/n}\nn\\
&\leq \frac{n^{L-2c|S|L/3n}}{L!}.\nn
\end{align}

{\bf Explanation for \eqref{B1a}:} Having chosen a set $X$ of $L$ vertices, we bound the probability that the set is contained in $B(p,S)$ by the probability that the vertices in $X$ each have at most $A$ neighbors in $S\setminus X$.

Thus, from the Markov inequality,
\[
\Pr(Z\geq n^{1-c|S|/4n})\leq \frac{\E\brac{\binom{Z}{L}}}{\binom{n^{1-c|S|/4n}}{L}} \leq \frac{n^{L-o(L)-2c|S|L/3n}}{n^{L-c|S|L/4n}}\leq n^{-c|S|L/3n}=o(n^{-r}).
\]

\item Let $\ell=c\log n/20$.
\begin{align*}
\Pr(\exists v,w\in\SMA:\dist(v,w)\leq 2) &\leq \sum_{j=2}^3 n^jp^{j-1}\brac{\sum_{i=0}^{\ell}\binom{n-j}{i}p^i(1-p)^{n-j-i}}^2 \\
&\leq \left(cn\log n + c^2 n \log^2 n\right) \cdot (n^{-2/3})^2=o(1).
\end{align*}

\item If this property fails then there is a connected set $T$ of at most $t_0=1+\frac{100rn}{c|S|}$ vertices that contains a set $T_1$ of size $t_1=\frac{10rn}{c|S|}$ vertices, each of which has at most $s_0=\frac{c|S|\log n}{20n}$ neighbors in $S\setminus T$. The probability of this can be bounded by
\begin{align*}
t_0\binom{n}{t_0}t_0^{t_0-2}p^{t_0-1}\binom{t_0}{t_1} \brac{\sum_{i=0}^{s_0}\binom{|S|}{i}p^i(1-p)^{|S|-t_0-i}}^{t_1} &\leq t_0n(c\log n)^{t_0}\brac{2\bfrac{|S|ep}{s_0}^{s_0}e^{-c|S|\log n/n}}^{t_1}\\
&=t_0n(c\log n)^{t_0}(2(20e)^{s_0}e^{-20s_0})^{t_1} \\
&\leq t_0n(c\log n)^{t_0}e^{-15s_0t_1}=o(n^{-r}).
\end{align*}

\item Proof of this can be found in Chapter 3 of \cite{book}.

\item Let $s_1=\frac{n(\log\log n)^2}{\log n}$. The probability of the existence of a pair of disjoint sets $S_1,S_2$ of size $s_1$ with no edge between them can be bounded by 
\[
\binom{n}{s_1}^2(1-p)^{s_1^2}\leq \bfrac{n^2e^{2}}{s_1^2e^{s_1p}}^{s_1}=o(n^{-r}).
\]
\item[{\bf B6,7}] These follow from standard first-moment calculations.
\end{enumerate}

\subsection{Proof of Lemma \ref{nlem}}\label{yy}
(a) This follows from Property {\bf B1}.

(b) This follows from Property {\bf B3}.

(c) This follows from Part (a) of Lemmas \ref{basic} and \ref{lemma1}.

\subsection{Proof of Lemma \ref{lem2}}\label{xx}
We first prove the following lemma bounding the edge density of small sets.
\begin{lemma}\label{dense}
In $G_{n,p}$, with $p \approx \frac{c\log n}{n}$, the following holds w.v.h.p.:
\begin{enumerate}[{\bf P1}]
\item For each $S\subseteq [n]$ satisfying $\log^{1/2} n \le |S|\leq n/\log^2 n$, we have that $e(S)\leq 3|S|$, where $e(S)$ denotes the number of edges contained in $S$.
\item For each $S\subseteq [n]$ satisfying $|S|\leq\r n$ with $\r \le 1/100$, we have that $e(S)\leq e\r^{1/2} c|S|\log n$. 
\end{enumerate}
\end{lemma}
\begin{proof}
In the respective cases, the probability that there exists a set $S$ with more edges than claimed can be bounded by
\begin{enumerate}[{\bf P1}]
\item \begin{align*}
\sum_{s=\log^{1/2}n}^{n/\log^2n}\binom{n}{s}\binom{\binom{s}{2}}{3s}p^{3s} &\leq \sum_{s=\log^{1/2}n}^{n/\log^2n}\brac{\frac{ne}{s}\cdot\bfrac{s^2ec\log n}{6sn}^3}^s \\
&\leq \sum_{s=\log^{1/2}n}^{\log n}\frac{1}{n^s} + \sum_{s=\log n}^{n/\log^2n}\bfrac{c^3}{2\log n}^s \\
&=o(n^{-r}).
\end{align*}
\item \begin{align*}
\sum_{s=e\r^{1/2} c\log n}^{\r n}\binom{n}{s}\binom{\binom{s}{2}}{e\r^{1/2} cs\log n}p^{e\r^{1/2} cs\log n} &\leq \sum_{s=e\r^{1/2} c\log n}^{\r n}\brac{\frac{ne}{s}\cdot\bfrac{s^2ec\log n}{2e\r^{1/2} csn\log n}^{e\r^{1/2} c\log n}}^s \\
&\leq \sum_{s=e\r^{1/2} c\log n}^{\r n}\brac{\bfrac{s}{n}^{1-2/e\r^{1/2} c\log n}\cdot\frac{1}{2\r^{1/2}}}^{ce\r^{1/2} s\log n} \\
&\leq \sum_{s=e\r^{1/2} c\log n}^{\r n}\brac{\bfrac{s}{n}^{1/2} \cdot\frac{1}{2\r^{1/2}}}^{ce\r^{1/2} s\log n} \\
&\leq \sum_{s=e\r^{1/2} c\log n}^{\r n}\brac{\frac{1}{2}}^{c^2e^2\r \log^2 n} \\
&=o(n^{-r}).
\end{align*}
\end{enumerate}
\end{proof}
Armed with Lemma \ref{dense}, we can proceed to the proof of Lemma \ref{lem2}.

(a) Suppose that there exists $S$ with $1 \le |S|\leq n/\log^4 n$ that does not satisfy Part (a) of Lemma \ref{lem2}. Let $T=N_i(S)$. Note that $|S\cup T| \ge \frac{\m_{min}\a_{min}\log n}{26}$ because the vertices of $G_i$, not in $A_{\bm}$, have degree at least $\frac{\m_{min}\a_{min}\log n}{26}$ (which was deduced in Section \ref{validation_of_step2}). Then $|S\cup T|\leq |S|(1+\m_{\min}\a_{\min}\log n/1000)\leq \m_{\min}\a_{\min}|S|\log n/999$ and from our bounds on degrees, that $e(S\cup T)\geq |S|\m_{\min}\a_{\min}\log n/52 > 3|S \cup T|$, contradicting Lemma \ref{dense}(P1), with $c=1$.

(b) Suppose now that $S$ is a set with $n/\log^4 n\leq |S|\leq \m_{\min}^2\a_{\min}^2n/10^7$ that does not satisfy Part (b) of Lemma \ref{lem2}. Let $T=N_i(S)$. Then $|S\cup T|\leq 4|S|\leq 4\m_{\min}^2\a_{\min}^2n/10^7$ and $e(S\cup T)\geq |S|\m_{\min}\a_{\min}\log n/52 \ge |S \cup T|\m_{\min}\a_{\min}\log n/208$, contradicting Lemma \ref{dense}(P2), with $c=1$ and $\r = 4\m_{\min}^2\a_{\min}^2/10^7$.

(c) Suppose that $S$ is an arbitrary connected component of the graph induced by color $i$ on the vertex set $V_i \setminus A_\bm$. If $|S|\le n/\log^4 n$, then by (a) we know that $|N_i(S)|\ge |S|\m_{min}\a_{min}\log n/1000$. However, by Lemma \ref{nlem}(b), we know that $|A_\bm \cap N_i(S)| \le 10r^2|S|/\a_{min}\m_{min}$. Thus, there are vertices $u\in S$ and $v \in V_i \setminus A_\bm$ such that $uv$ is an $i$-colored edge, contradicting the assumption that $S$ is a connected component. Thus, we can assume that every connected component of the induced graph using edges of color $i$ on $V_i \setminus A_\bm$ has size more than $n/\log^4 n$. It follows from (b) and Lemma \ref{nlem}(a) that every component has size at least $s_0=\m_{\min}^2\a_{\min}^2n/10^7$. Now apply Property {\bf B5} with $c=\a_{\min}$ to show that there cannot be two such large components.

\subsection{Proof of Lemma \ref{addlemma}}\label{xxx}
Part (a) of Lemma \ref{addlemma} follows from Part (b) of Lemmas \ref{basic} and \ref{lemma1}. Part (b) of Lemma \ref{addlemma} follows easily from the fact that the graph induced by the color $c_r$ has the distribution $G_{n,p'}$, where $p'=\a_r \cdot p_1 \ge \frac{\log n + \log\log n + \om/2}{n}$.

\subsection{Proof of Lemma \ref{BKS}}\label{pbks}
Connectivity follows as in the proof of Part (c) of Lemma \ref{lem2} in Section \ref{xx}. The other condition follows from Property {\bf B5}.

\section{Concluding remarks}\label{ConcRem}
The ultimate goal is to understand the thresholds for the existence of varying patterns in edge-colored random graphs. The hardest question seems to be to find the threshold for the existence of arbitrary patterns. Periodic patterns were dealt with in \cite{AF} and \cite{EFK}.

Leaving this problem aside, we can still ask for the likely value of $hcp(G_{n,p})$ for all values of $p$ between the threshold for Hamiltonicity and the value in Theorem \ref{th2}.

{\bf Acknowledgement:} We are grateful to a reviewer for pointing to an error in a previous version and also for an excellent review of our paper.

\end{document}